\documentclass[11pt,a4paper]{amsart}
\usepackage{graphicx}
\usepackage{amssymb}

\usepackage[square, numbers, comma]{natbib}

\usepackage[normalem]{ulem}

\usepackage{amsmath}
\usepackage{amsthm}
\usepackage{dsfont}
\usepackage{a4wide}
\usepackage{pdfsync}
\usepackage{hyperref}
\usepackage{enumerate}
\usepackage{color}
\usepackage{cancel}

\DeclareMathOperator*{\esssup}{ess\,sup}
\DeclareMathOperator*{\essinf}{ess\,inf}
\DeclareMathOperator*{\esslim}{ess\,lim}

\newtheorem{theorem}{Theorem}[section]
\newtheorem{proposition}[theorem]{Proposition}

\newtheorem{lemma}[theorem]{Lemma}

\theoremstyle{remark}
\newtheorem{remark}[theorem]{Remark}

\theoremstyle{definition}
\newtheorem{definition}[theorem]{Definition}

\newcommand{\sgn}{\,{\rm sgn}}

\def\eps{\varepsilon}

\def\rr{\mathbb{R}}

\def\eps{\varepsilon}

\def\ml{\mathcal{L}}

\def\ou{\overline u}

\title[Heat equations with fast convection ]{Heat equations with fast convection: source-type solutions and large-time behaviour}
\author[J.~Endal]{J{\o}rgen Endal}
\address[J.~Endal]{Department of Mathematical Sciences\\
Norwegian University of Science and Technology (NTNU)\\
N-7491, Trondheim, Norway}
\email[]{jorgen.endal\@@{}ntnu.no}
\urladdr{https://folk.ntnu.no/jorgeen/}

\author[L.~I.~Ignat]{Liviu I. Ignat}
\address[L.~I.~Ignat]{Institute of Mathematics ``Simion Stoilow'' of the Romanian Academy\\
21 Calea Grivitei Street, 010702 Bucharest, Romania\\
and\\
The Research Institute of the University of Bucharest - ICUB\\
University of Bucharest\\
90-92 Sos. Panduri, 5th District, Bucharest, Romania
}
\email[]{liviu.ignat\@@{}gmail.com}
\urladdr{http://www.imar.ro/~lignat}

\author[F.~Quir\'{o}s]{Fernando Quir\'{o}s}
\address[F.~Quir\'{o}s]{Departamento de Matem\'aticas\\
Universidad Aut\'onoma de Madrid (UAM)\\
Campus de Cantoblanco, 28049 Madrid, Spain\\
and\\
Instituto de Ciencias Matem\'aticas ICMAT (CSIC-UAM-UCM-UC3M)\\
28049 Madrid, Spain}
\email[]{fernando.quiros\@@{}uam.es}
\urladdr{https://matematicas.uam.es/~fernando.quiros}

\begin{document}

\begin{abstract}
We study the existence and uniqueness of  source-type solutions to the Cauchy problem for the heat equation with fast convection under certain tail control assumptions.  We allow the solutions to change sign, but we will in fact show that they have the same sign as the initial data, which is a multiple of the Dirac delta. As an application, we obtain the large-time behaviour of nonnegative bounded solutions with integrable initial data to heat equations with fast convection, covering the case of several dimensions that remained open since the end of last century.
\end{abstract}

 %Diffusion processes
%35R09, %Integro-partial differential equations
%45K05, %Integro-partial differential equations
%46B50. %Compactness in Banach (or normed) spaces

\maketitle

\tableofcontents

\begin{center}
	\emph{Dedicated to professor Yihong Du on the occasion of his 60th birthday,  with thanks for his friendship and admiration for his mathematics.}
\end{center}

%%%%%%%%%%%%%%%%%%%%%%%%%%%%%%%%%%%%%%%%%%%%%%%%%%%%%%%%%%
%%%%%%%%%%%%%%%%%%%%%%%%%%%%%%%%%%%%%%%%%%%%%%%%%%%%%%%%%%
\section{Introduction and main results}
\label{sect-Introduction} \setcounter{equation}{0}

\subsection{Source-type solutions}
The first goal of this paper is to study the existence and uniqueness of solutions to the diffusion problem with convection
\begin{equation}
    \label{eq:source.type}
    \partial_t u+\mathcal{L} u + \partial_{x_N} (|u|^{q-1}u)=0
    \quad \text{in }Q:=(0,\infty)\times \rr^N,\qquad u(0)=M\delta_0,\ M\neq0 \quad \text{on }\rr^N,
\end{equation}
in the \emph{fast convection} range $q\in (1-1/N,1)$, both for $\mathcal{L}=-\Delta$, and for $\mathcal{L}=-\Delta_{x'}$, where $\Delta_{x'}$ stands for the Laplacian in the first $N-1$ coordinates. Here  $\delta_0$ denotes the Dirac mass located at the origin. Solutions to this problem, with a multiple of a Dirac delta as initial datum,  are known as \emph{source-type} solutions, and also as~\emph{fundamental} solutions.
  We are interested in solutions that conserve mass along the evolution,
$$
\int_{\mathbb{R}^N} u(\cdot,t) =M\quad\text{for all }t>0.
$$
That is where the restriction for $q$ from below, $q>1-1/N$, comes from.

We start by considering the case of solutions that have a definite sign. We assume, without loss of generality, that they are nonnegative, and  prove uniqueness in the framework of entropy solutions; see Section~\ref{sect-Entropy} for the definition and main properties of solutions of this kind. This is the content of our first theorem.

\begin{theorem}
\label{thm:uniqueness.non-negative}
Let $q\in (1-1/N,1)$ and  $\mathcal{L}=-\Delta$ or $\mathcal{L}=-\Delta_{x'}$. There exists at most one nonnegative entropy solution  $u$ of  problem~\eqref{eq:source.type}.
\end{theorem}

\begin{remark}
When $\mathcal{L}=-\Delta_{x'}$, the invariance of problem~\eqref{eq:source.type} under certain scalings and the uniqueness result imply that    source-type solutions have a self-similar form,
 $$
u(t,x',x_N)=t^{-\alpha}f(x't^{-1/2},x_Nt^{-\beta}),\qquad \alpha=\frac {N+1}{2q},\quad\beta=\frac{N+1-q(N-1)}{2q};
$$
see~\cite{EVZIndiana} for a similar argument when $q>1$. If $q\ge (N+1)/(N-1)$, $\beta\le0$, which is incompatible with the initial datum being a multiple of a Dirac mass.
\end{remark}

 When $q>1$ the above analysis was performed first for the rather special case $N=1$ and $\mathcal{L}=-\Delta_{x'}$,  in which problem~\eqref{eq:source.type} reduces to the scalar conservation law
\begin{equation}
\label{eq:scalar.conservation.law}
\partial_t u+\partial_{x} (|u|^{q-1}u)=0,
\end{equation}
in~\cite{MR735207},
later for $N=1$ and $\mathcal{L}=-\Delta$ in~\cite{EVZArma}, and finally for both operators in any dimension in~\cite{EVZIndiana}.

In the fast convection range, up to now only the special case $N=1$, $\mathcal{L}=-\Delta_{x'}$ has been covered; see~\cite[Proposition 1.1]{LaurencotFast}. The proofs when $q>1$  use that the convection nonlinearity is $C^1$, something that is not true in the fast convection range; hence the need of a new approach.  Our proof is inspired by our recent paper~\cite{JFL}, where  we consider a nonlocal analogue of~\eqref{eq:source.type} in which the $\mathcal{L}$ is a stable nonlocal diffusion operator.

We will next allow for sign changes. However, we will prove  that under certain tail-control assumption, source-type solutions have the same sign as the constant $M$ in the initial datum. Hence, by the previous step, we have uniqueness in this class.

\begin{theorem}
\label{thm:uniqueness.no.sign}
Let $q\in (1-1/N,1)$,   $\mathcal{L}=-\Delta$ or $\mathcal{L}=-\Delta_{x'}$, and $M>0$. Then, any entropy solution $u\in L^\infty((0,\infty);L^1(\rr^N))$ of  problem~\eqref{eq:source.type} such that
\begin{equation}
    \label{eq:tail.extra.hyp}\int_{|x|>r}|u(t,x',x_N)|\,{\rm d}x'{\rm d}x_N\rightarrow 0 \text{ as }t\rightarrow 0^+\text{ for any }r>0,
\end{equation}
is nonnegative. As a corollary, there is at most one  entropy solution to~\eqref{eq:source.type} satisfying~\eqref{eq:tail.extra.hyp}.
\end{theorem}

For $q\in(1,(N+1)/(N-1))$, $q\le 2$,  a similar analysis
was performed  in~\cite{Carpio1996}.
No such results were available up to now for the fast convection range.

\begin{remark} (a) The uniqueness results in~\cite[Theorem 1.1]{MR735207} for the scalar conservation law~\eqref{eq:scalar.conservation.law}, corresponding to the special case $N=1$,  $\mathcal{L}=-\Delta_{x'}$, which are valid  for any $q>1$, do not assume a sign for the solution. Besides, they are valid for solutions having a finite measure with a sign (which need not be a multiple of a Dirac mass) as initial datum.

\noindent (b) As mentioned above, there are no nonnegative solutions to~\eqref{eq:source.type} when $\mathcal{L}=-\Delta_{x'}$  if $q\ge(N+1)/(N-1)$.  However, this nonexistence result was based on the uniqueness of source-type solutions. Since the question of uniqueness of source-type solutions that may change sign is still open in this range of exponents (even under the tail-control assumption), we may not yet exclude the possibility of their existence.

\noindent (c) In dimension $N=1$ our proof of Theorem~\ref{thm:uniqueness.no.sign} is valid for $q\in(1,3)$ when $\mathcal{L}=-\Delta$.
Hence we are able to improve the available results, covering the range $q\in(2,3)$; see Remark~\ref{rk:q.2.3}.
\end{remark}

\begin{theorem}
\label{thm:existence}
Let $q\in (1-1/N,1)$ and  $\mathcal{L}=-\Delta$ or $\mathcal{L}=-\Delta_{x'}$. There exists a  nonnegative entropy solution  $u$ of  problem~\eqref{eq:source.type} satisfying the tail-control condition~\eqref{eq:tail.extra.hyp}.
\end{theorem}

Existence of solutions with a tail control for  $q>1$ in dimension $N=1$ was proved in~\cite{EVZArma}  under the restrictions $q\in (1,2)$ if $\mathcal{L}=-\Delta_{x'}$ and $q\in(1,3)$ if $\mathcal{L}=-\Delta$. Higher dimensions $N>1$ were covered in~\cite{EVZIndiana}, with the restrictions $q\in (1,(N+1)/(N-1))$, $q\le 2$.

%%%%%%%%%
\subsection{Application to  large-time behaviour}

Our interest in solutions to~\eqref{eq:source.type} stems from the role they play in the description of the large-time behaviour of bounded solutions to
\begin{equation}
\label{eq:main}
  \partial_tu-\Delta u+\partial_{x_N}(|u|^{q-1}u)=0\quad\text{in }Q, \qquad u(0)=u_0\quad\text{on }\mathbb{R}^N,
\end{equation}
with initial datum $u_0\in L^1(\mathbb{R}^N)$. Indeed, in the slow convection range $q>1$, the large-time behaviour of solutions to this problem is given, if $\displaystyle M=\int_{\mathbb{R}^N} u_0\neq 0$, by
\begin{equation}
\label{eq:convergence}
t^{\max\{\frac N2, \frac{N+1}{2q}\}(1-\frac 1p)}\|u(t)-U(t)\|_{L^p(\mathbb{R}^N)}\to0\quad\textup{as }t
\to\infty
\end{equation}
for all $p\in[1,\infty)$, where $U$  is the unique very weak (if $q\ge q_*:=1+\tfrac1N$) or entropy (if $q<q_*$) source-type solution  in $Q$ with mass $M$ of
\begin{align}
    \label{eq:diffusion}
    \partial_tU-\Delta U  =0\quad&\textup{if }q>q_*,\\
    \label{eq:critical}
    \partial_tU-\Delta U + \partial_{x_N} (|U|^{q-1}U)=0\quad&\textup{if } q=q_*,\\
    \label{eq:convection}
    \partial_tU-\Delta_{x'}{U} + \partial_{x_N}( |U|^{q-1}U) =0\quad&\textup{if } q\in(1,q_*),
\end{align}
as proved in the remarkable series of papers~\cite{EZ,EVZArma,EVZIndiana, EsZu97} for nonnegative solutions and in~\cite{Carpio1996} for solutions with sign changes;  see also~\cite{Zua20}.

The only result on long-term behaviour for solutions to~\eqref{eq:main} with $q<1$ known to date was given in~\cite{LaurencotFast}, where the author shows in the one-dimensional case $N=1$, assuming also that the solution is nonnegative,  that for all $q\in(0,1)$ the function giving the limit is a source-type solution to~\eqref{eq:convection}, as expected, since $1<q_*$. No results for $q<1$ were available in several dimensions, the main obstruction being the lack of a uniqueness theory for source-type solutions to the limit problem. Our uniqueness results will allow us to overcome this difficulty, proving the large-time behaviour for nonnegative solutions in any dimension if $q\in (1-1/N,1)$. We recall that the restriction from below for $q$ is not technical, since it is required to  guarantee that the mass is conserved along the evolution,
$$
\int_{\mathbb{R}^N} u(\cdot,t) =\int_{\mathbb{R}^N} u_0\quad\text{for all }t>0.
$$

\medskip

\begin{theorem}
\label{thm:long.time.behaviour}
Let $q\in (1-1/N,1)$ and  $\mathcal{L}=-\Delta$ or $\mathcal{L}=-\Delta_{x'}$.
Let  $u\in C([0,\infty);L^1(\rr^N))$   be a nonnegative  bounded entropy solution to~\eqref{eq:main} with $u_0\in L^1(\mathbb{R}^N)\cap L^\infty(\mathbb{R}^N)$. Let $U$ be the unique source-type solution to~\eqref{eq:source.type} with $M=\int_{\mathbb{R}^N}u_0$. Then, for all $p\in[1,\infty)$,
$$
t^{\frac{N+1}{2q}(1-\frac 1p)}\|u(t)-U(t)\|_{L^p(\mathbb{R}^N)}\to0\quad\textup{as }t
\to\infty.
$$
\end{theorem}
\begin{remark}
    The convergence result yields an alternative existence proof of a source-type solution  satisfying the tail-control assumption~\eqref{eq:tail.extra.hyp} for the case $\mathcal{L}=-\Delta_{x'}$.
\end{remark}

The (parabolic and hyperbolic) estimates needed for compactness can be obtained as in paper~\cite{JFL}.
We will give a brief sketch in Section~\ref{sec:AsymptoticBehaviour}.  We point out that such compactness arguments are not valid for solutions with sign changes.

\subsection{Some open problems}

\noindent\textsc{Existence for $q\in (2,3)$ in dimension $N=2$. } As we have mentioned, the existence results for source-type solutions in \cite{EVZIndiana} for $N>1$ have the restrictions $q<(N+1)/(N-1)$, $q\le 2$. The first one has an essential nature, at least when $\mathcal{L}=-\Delta_{x'}$. But the second one, stemming from the hyperbolic techniques used in the proof, might be technical. Hence, it would be interesting whether there are source-type solutions for $q\in (2,3)$ in dimension $N=2$, or, on the contrary, if they do not exist.

\noindent\textsc{Non-existence for $q\ge (N+1)/(N-1)$. } Our guess would be that no source-type solutions exist for $q\ge (N+1)/(N-1)$. However, existence has not yet been excluded in this range  when $\mathcal{L}=-\Delta$, or when $\mathcal{L}=-\Delta_{x'}$ if sign changes are allowed, even under some tail-control assumption.

\noindent\textsc{Long-time behaviour for solutions with sign changes. } The compactness arguments that we are using here to obtain the large-time behaviour for nonnegative solutions, borrowed from~\cite{JFL}, are not able to deal with the nonlinear term for solutions that change sign. New compactness arguments are hence required. Nevertheless, if compactness were available, the required uniqueness for the expected limit is already provided by the present paper.

In the one dimensional case, the method used in \cite{cazacu} to obtain the long time behaviour of changing sign solutions, based on the compensated-compactness techniques from \cite{tartar}, cannot be used here, since the nonlinearity is not $C^1$. The issue of how to extend the  compensated-compactness method in the context of continuous flows remains to be investigated.

\noindent\textsc{The case with integral 0. } If we were able to prove~\eqref{eq:convergence} for solutions with sign changes, in the case $M=\int_{\mathbb{R}^N}u_0 =0$, the result would only say that
$$
\|u(\cdot,t)\|_{L^p(\mathbb{R}^N)}=o\big(t^{\max\{\frac N2, \frac{N+1}{2q}\}(1-\frac 1p)}) \  \text{as}\ t\rightarrow \infty,
$$
and would give no information about the asymptotic shape of the solution. The first step to understand the large-time behaviour in this interesting case would be to obtain a sharp decay rate for the solution.

\subsection{Organization of the paper} The rest of the paper is organized as follows. We gather in Section~\ref{sect-Entropy} some preliminary material on entropy solutions (definition, properties, estimates) that will be required later. Section~\ref{sec:Uniqueness.nonnegative} is devoted to prove the uniqueness of nonnegative source-type solutions. In Section~\ref{sect-SignChanges} we prove that source-type solutions that may change sign, do not so if they satisfy the tail-control property~\eqref{eq:tail.extra.hyp}. Existence of source-type solutions is also discussed here. Finally, as an application of our uniqueness results, we obtain the large time behaviour of nonnegative solutions to problem~\eqref{eq:main} with  {bounded} and integrable initial data in Section~\ref{sec:AsymptoticBehaviour}.

%%%%%%%%%%%%%%%%%%%%%%%%%%%%%%%%%%%%%
\section{Entropy formulation}
\label{sect-Entropy} \setcounter{equation}{0}

In this section we present the entropy framework for the problem
\begin{equation}\label{eq.fLgeneral}
\partial_tu+\mathcal{L} u+\partial_{x_N}(f(u))=0\quad\text{in }Q, \qquad u(0)=u_0\quad\text{on }\mathbb{R}^N,
\end{equation}
where $f(u)=|u|^{q-1}u$, $1-\frac 1N<q\neq 1$, $\mathcal{L}$ is either $-\Delta$ or $-\Delta_{x'}$, and $u_0$ is at least a finite measure, gathering the preliminary material concerning solutions and sub- and supersolutions that will be required later.

\subsection{Entropy solutions}

We start with the definition of entropy solution.

\begin{definition}[Entropy solution]\label{def.entropySolution}
A function $u$ is an \emph{entropy solution} of \eqref{eq.fLgeneral} if:
\begin{enumerate}[{\rm (a)}]
\item \textup{(Regularity)} $u\in   {L^\infty_\textup{loc}((0,\infty);L^\infty(\rr^N))}$ and $\nabla u\in L_\textup{loc}^2(Q)$.
\item \textup{(Entropy inequality)} For all $k\in\rr$ and all $0\leq \phi\in C_\textup{c}^\infty(Q)$,
\begin{equation}\label{eq:entropy.inequality}
\iint_Q \Big( |u-k|\partial_t\phi + \sgn(u-k)\big(f(u)-f(k)\big)\partial_{x_N}\phi-|u-k|(\mathcal{L}\phi) \Big)\geq0.
\end{equation}
\item\textup{(Initial data in the sense of finite measures)} For all $\psi\in C_\textup{b}(\rr^N)$,
\begin{equation}
	\label{eq:initial.data}
\esslim _{t\to0^+}\int_{\rr^N}u(t)\psi= \int_{\rr^N}\psi\,{\rm d}u_0.
\end{equation}
\end{enumerate}
\end{definition}

\begin{remark}\label{regularity.issue}
\begin{enumerate}[{\rm (a)}]
\item If $u_0 \in L^1(\rr^N)\cap L^\infty(\rr^N)$, condition~\eqref{eq:initial.data} is implied if the data is taken in an $L^1$-sense. It is then standard to show that instead of assuming the $L^1$-continuity at $t=0$, we can add the term $\int_{\rr^N}|u_0-k|\phi(0)$ on the left-hand side of the entropy inequality~\eqref{eq:entropy.inequality}  and take $\phi\in C_{\rm c}^\infty(\overline{Q})$.

\item It is also straightforward to show that classical solutions of \eqref{eq.fLgeneral} are entropy solutions, and that entropy solutions are \emph{very weak solutions} of that problem, i.e.:
$$
 \iint_Q \big(u \partial_t\phi+f(u)\partial_{x_N}\phi - u(\mathcal{L}\phi)\big)+\int_{\rr^N}\phi(0)\,{\rm d}u_0=0 \quad\text{for all }\phi\in C_\textup{c}^\infty(\overline{Q}).
$$
\end{enumerate}
\end{remark}

Let us now collect some more or less standard results for entropy solutions. The proof can be deduced by following \cite{panov2020} or \cite{AnBr20} with $\alpha=2$.

\begin{lemma}\label{lem.UniquenessPropertiesEntropy}
{\rm (a)} Assume $p\in[1,\infty)$ and $u_0\in L^\infty(\rr^N)$. Then there is \emph{at least} one entropy solution $u\in L^\infty(Q)\cap C([0,\infty);L_\textup{loc}^p(\rr^N))$ and $\nabla u\in L_{\textup{loc}}^2(Q)$ of~\eqref{eq.fLgeneral} with initial data $u_0$. Moreover, let $u,\bar{u}$ be entropy solutions of~\eqref{eq.fLgeneral} with initial data $u_0,\bar{u}_0$, respectively. Then:
\begin{enumerate}[{\rm (i)}]
\item\textup{(Comparison principle)} If $u_0\leq \bar{u}_0$ a.e.~on $\rr^N$, then $u\leq \bar{u}$ a.e.~in $Q$.
\item\textup{($L^1$-contraction)}  If $u_0-\bar{u}_0\in L^1(\rr^N)$, then, for a.e.~$t>0$,
$$
\int_{\rr^N}|u(t)-\bar{u}(t)|\leq \int_{\rr^N}|u_0-\bar{u}_0|.
$$
\item\textup{($L^\infty$-bound)} For a.e.~$(t,x)\in Q$,  $\essinf_{x\in\rr^N}u_0(x)\leq u(t,x)\leq\esssup_{x\in\rr^N}u_0(x)$.
\item\textup{($L^1$-bound)} For a.e.~$t>0$, $\|u(t)\|_{L^1(\rr^N)}\leq \|u_0\|_{L^1(\rr^N)}$.
\item\textup{(Mass conservation)} For a.e.~$t>0$,
$$
\int_{\rr^N}u(t)=\int _{\rr^N}u_0.
$$
\end{enumerate}
\noindent{\rm (b)} Assume, in addition, $u_0\in L^1(\rr^N)\cap L^\infty(\rr^N)$. Then there is \emph{at most} one entropy solution $u$ of~\eqref{eq.fLgeneral} with initial data $u_0$, and moreover, $u\in C([0,\infty);L^1(\rr^N))$.
\end{lemma}

\begin{remark}\label{CL1}
By an approximation argument, we obtain that for general initial data which are just finite  measures, $u\in C((0,\infty);L^1(\rr^N))$ and $\int_{\mathbb{R}^N}u(t)=\int_{\mathbb{R}^N}{\rm d}u_0$ for all  $t>0$.
\end{remark}

As in \cite{JFL}, we also need some further regularity properties whose proofs follow the lines of the mentioned paper except that $\alpha=2$.

\begin{lemma}\label{lem.FutherPropEntropy}
Assume $0\leq u_0\in L^1(\rr^N)\cap L^\infty(\rr^N)$. Then:
\begin{enumerate}[{\rm (a)}]
\item\textup{(Energy estimate)}  For all $0<\tau<T<\infty$,
\begin{equation*}
\int_\tau^T \|\mathcal{L}^{\frac{1}{2}}u(s)\|_{L^2(\rr^N)}^2\,{\rm d}s\leq\frac{1}{2}\|u(\tau)\|^2_{L^2(\rr^N)}.
\end{equation*}

\item\textup{(Preliminary estimate on the time derivative)} For all smooth bounded domain $\Omega\subset \rr^N$,
$$
\partial_t u\in L^2((0,\infty);H^{-1}(\Omega)).
$$

\item\textup{(Preliminary tail control)} Let $\rho_R(x):=\rho(x/R)$, where $0\leq \rho\leq 1$ is a $C^\infty$-functions such that $\rho\equiv 0$ if $|x|\leq 1$ and $\rho\equiv 1$ if $|x|>2$. Then, for all $R>0$ and $t>0$,
\begin{equation*}
\label{tail.control}
\begin{split}
&\int_{|x|>2R}u(t)\lesssim \int_{|x|> R}u_0+M t\|\mathcal{L}\rho_R\|_{L^\infty(\rr^N)} +
\int _0^t \int_{\rr^N} u^q|\partial_{x_N}\rho_R|.
\end{split}
\end{equation*}
\end{enumerate}
\end{lemma}

\begin{remark}
\begin{enumerate}[{\rm (a)}]
\item When $\mathcal{L}=-\Delta_{x'}$,
$$
\|\mathcal{L}^{\frac{1}{2}}u(s)\|_{L^2(\rr^N)}^2=\int_{\rr^N}|\nabla_{x'}u(s)|^2,
$$
which gives us no control on the derivative in the direction $x_N$.
\item By an approximation argument, we obtain that for $u_0=M\delta_{0}$,
$$
\int_{|x|>2R}u(t)\lesssim M t\|\mathcal{L}\rho_R\|_{L^\infty(\rr^N)} +
\int _0^t \int_{\rr^N} u^q|\partial_{x_N}\rho_R|.
$$
\end{enumerate}
\end{remark}

\subsection{Entropy subsolutions}
In this subsection, we collect results on entropy sub- and supersolutions from \cite{panov2020, panov2022}. We denote the respective positive and negative parts by $s^+=\max\{s,0\}$ and $s^-=-\min\{s,0\}=\max\{-s,0\}$. So $s=s^+-s^-$ and $|s|=s^++s^-$.

\begin{definition}[Entropy subsolution]\label{def.sub-entropySolution}
A function $u$ is an \emph{entropy subsolution} of~\eqref{eq.fLgeneral} if:
\begin{enumerate}[{\rm (a)}]
\item \textup{(Regularity)} $u\in  L^\infty (Q)$ and $\nabla u\in L^2_{\textup{loc}}(Q)$.

\item \textup{(Entropy inequality)} For all $k\in\rr$ and all $0\leq \phi\in C_\textup{c}^\infty(Q)$,
\begin{equation}\label{eq:sub.entropy.inequality}
\iint_Q \Big( (u-k)^+\partial_t\phi + \sgn^+(u-k)\big(f(u)-f(k)\big)\partial_{x_N}\phi-(u-k)^+(\mathcal{L}\phi) \Big)\geq0.
\end{equation}
\item\textup{(Initial data)} \begin{equation}
	\label{eq:sub.initial.data}
\esslim _{t\to0^+} (u(t,x)-u_0(x))^+=0 \quad \text{in $L^1_{\textup{loc}}(\rr^N)$.}
\end{equation}
\end{enumerate}
\end{definition}

As before, it is possible to define the above concepts as one integral inequality \cite[Prop.~2.1]{panov2022}: For all $k\in\rr$ and all $0\leq \phi\in C_\textup{c}^\infty(\overline{Q})$
\begin{equation}
\label{eq:sub.entropy.inequality.full}
  \iint_Q \Big( (u-k)^+\partial_t\phi + \sgn^+(u-k)\big(f(u)-f(k)\big)\partial_{x_N}\phi-(u-k)^+(\mathcal{L}\phi) \Big)+\int_{\rr^N}(u_0-k)^+\phi(0)\geq0.
\end{equation}

In a similar way, a function $u$ is called an \emph{entropy supersolution} if
\[
\iint_Q \Big( (k-u)^+\partial_t\phi + \sgn^+(k-u)\big(f(k)-f(u)\big)\partial_{x_N}\phi-(k-u)^+(\mathcal{L}\phi)\Big)+\int_{\rr^N}(k-u_0)^+\phi(0)\geq0.
\]

Later the below properties are used on $Q$.

\begin{proposition}\label{results.subsolutions}
The following holds for equation \eqref{eq.fLgeneral}:
\begin{enumerate}[{\rm (a)}]
\item A function $u$ is an entropy solution iff it is an entropy sub- and supersolution.

\item All entropy subsolutions satisfy
\[
\partial_tu+\mathcal{L} u+\partial_{x_N}f(u)\leq 0 \quad \text{in $\mathcal{D}'(Q)$}.
\]

\item {If $u_1$ and $u_2$ are entropy subsolutions, then $u=\max\{u_1,u_2\}$ is also an entropy subsolution.}	

\item If there exists a unique entropy solution $u$, then all entropy sub- and supersolutions $u_{-},u_{+}$ satisfy
\[ 	
u_{-}\leq u \leq u_{+} \quad \text{a.e. in $Q$.}
\]
\end{enumerate}
\end{proposition}

\begin{proof}
The first four properties follow from  \cite{panov2020,panov2022}. Also \cite[Theorem~1.1]{panov2022} shows that the biggest/smallest entropy solution is  also the biggest/smallest entropy sub-/supersolution.  When the uniqueness of the entropy solution $u$ is guaranteed, it satisfies the required inequality.
\end{proof}
%%%%%%%%%%%%%%%%%%%%%%%%%%%%%%%%%%%%%%%%%%%%%%%%%%%%%%%%%%%%%%%%%%%%%%%%%%%%%%%%%%
\section{Uniqueness of nonnegative source-type solutions}
\label{sec:Uniqueness.nonnegative}
\setcounter{equation}{0}

The goal of this section is to prove the uniqueness of nonnegative solutions to~\eqref{eq:source.type} in the fast convection range.
The proof uses some ideas from the recent paper~\cite{JFL}, in which the authors consider nonlocal diffusion problems with convection.

\begin{proof}[Proof of Theorem~\ref{thm:uniqueness.non-negative}]
We divide the proof into several steps. Assume that~\eqref{eq:source.type} has two nonnegative entropy solutions $u$ and $\ou$ with the same initial data $M\delta_0$.  In view of Remark \ref{CL1}, the solutions also belong to $L^\infty((0,\infty);L^1(\rr^N))\cap C((0,\infty);L^1(\rr^N))$.

\smallskip
\noindent(i) \emph{Integrating in the direction $x_N$.} It is easy to check that
\begin{equation}\label{integrated.u}
v(t,x'):=\int_{\rr}u(t,x',x_N)\,{\rm d}x_N \quad\textup{for   }(t,x')\in (0,\infty)\times \rr^{N-1}
\end{equation}
belongs to $C((0,\infty);L^1(\rr^{N-1}))$ and is the unique very weak solution of $\partial_tv=\Delta_{x'} v$ in $(0,\infty)\times \rr^{N-1}$ with initial data $M\delta_{0'}$ in the sense of  finite  measures. Thus, $v=M\Gamma_{N-1}$, where $\Gamma_{N-1}$ is the unique fundamental solution of the equation with mass 1, which has a self-similar form,
\[
\Gamma_{N-1}(t,x')=t^{-\frac{N-1}2}F(t^{-\frac 12}x'), \qquad F(\xi)=\frac{\textrm{e}^{-|\xi|^2/4}}{(4\pi)^{\frac{N-1}{2}}}.
\]
 After changing $u$ and $\overline{u}$ on a set of measure zero we then    have:
\begin{equation*}
\label{both.Nmass}
  \int_{\rr} u(t,x',x_N)\,{\rm d}x_N=  \int_{\rr} \ou(t,x',x_N)\,{\rm d}x_N=M\Gamma_{N-1}(t,x') \quad\textup{ {for all }}(t,x')\in (0,\infty)\times \rr^{N-1}.
\end{equation*}

\smallskip
\noindent(ii) \emph{Approximation by entropy solutions with initial data in $L^1(\rr^N)\cap L^\infty(\rr^N)$.} We start by constructing the initial data for our approximations.
Inspired by~\cite{EVZIndiana}, for each $r>0$, we define
\[
\begin{gathered}
\varphi_r(t,x')=\int_{-r}^r u(t,x',x_N)\,{\rm d}x_N-M\Gamma_{N-1}(t,x')=-\int_{|x_N|>r} u(t,x',x_N)\,{\rm d}x_N,
\\
u^r_{0,n}(x',x_N):=\Big(u(1/n,x',x_N)-\frac 1{2r}\varphi_r(1/n,x')\Big)\mathds{1}_{(-r,r)}(x_N).
\end{gathered}
\]
Notice that $ -M\Gamma_{N-1}(t,x')\leq \varphi_r(t,x')\leq 0$. Hence, the functions $u^r_{0,n}$ are nonnegative and  bounded,
\[
\|u^r_{0,n}\|_{L^\infty(\rr^{N})}\leq \| u(1/n)\|_{L^\infty(\rr^N)}+\frac{M}{2r}\|\Gamma_{N-1} (1/n)\|_{L^\infty(\rr^{N-1})}.
\]
On the other hand, they have integral $M$ since
\begin{equation}
	\label{mass.u_0n}
\int_{\rr} u^r_{0,n}(x',x_N)\,{\rm d}x_N=M\Gamma_{N-1}(1/n,x')\quad\text{for all }x'\in \rr^{N-1}.
\end{equation}
Let $u_n$ (we omit the dependence on the parameter $r$ for simplicity) be the entropy solution of~\eqref{eq.fLgeneral}  with initial data $u_{0,n}^r$. Let us prove that,   for any $t>0$, the sequence  $u_n(t)$ converges to $u(t)$ in $L^1(\rr^N)$ as $n\to\infty$. To this aim we observe that both $u_n$ and $u(\cdot +1/n)$ are entropy solutions of~\eqref{eq.fLgeneral},  whose respective initial data, $u_{0,n}^r$ and $u(1/n)$, belong to $L^1(\rr^N)\cap L^\infty(\rr^N)$. Then the $L^1$-contraction property shows that
\begin{align*}
\big\|u_n(t)-u\big(t+1/n\big)\big\|_{L^1(\rr^N)}&\leq \| u^r_{0,n}-u(1/n)\|_{L^1(\rr^N)}\\
&\leq
\int_{\rr^{N-1}}\bigg(\int_{|x_N|>r} u(1/n,x',x_N)\,{\rm d}x_N+ |\varphi_r(1/n,x')| \bigg)\,{\rm d}x'\\
&\leq 2\int_{\rr^{N-1}}\int_{|x_N|>r} u(1/n,x',x_N)\,{\rm d}x_N{\rm d}x' \rightarrow 0 \quad\text{as  $n\rightarrow\infty$.}
\end{align*}
Since $u\in C((0,\infty);L^1(\rr^N))$, we have that $u(t+1/n)$ converges to $u(t)$ in $L^1(\rr^N)$ as $n\rightarrow \infty$. Therefore, $u_n(t)\rightarrow u(t)$ in $L^1(\rr^N)$ as $n\rightarrow \infty$. The contraction property implies  that $u_n\rightarrow u$ in $C([\tau, \infty);L^1(\rr^N))$ as $n\rightarrow \infty$ for all  $\tau>0$.

Defining $\ou _{0,n}^r$  and then  $\ou_n$ in a similar way, we get by the same reasoning that $\ou_n(t)\rightarrow \ou(t)$ in $L^1(\rr^N)$ as $n\rightarrow \infty$.

\smallskip

\noindent (iii) \emph{Approximation of the nonlinearity. } We consider the solution of the problem \eqref{eq.fLgeneral} with initial data $u_{0,n}^r$ and nonlinearity
\begin{equation}
\label{f.eta}
f_\eta(u)=(u^2+\eta)^{q/2}-\eta^{q/2},\quad \eta>0.
\end{equation}
The corresponding solution will depend on  $\eta$. We denote it by $u^\eta_n$. Note that $f_\eta$ is $C^\infty(\rr)$ and satisfies the inequality
\begin{equation}
\label{ineq.f.eta}
 0\leq  f_\eta(s)=(s^2+\eta)^{q/2}-\eta^{q/2}\leq s^q\quad \text{for all $s\geq 0$.}
\end{equation}

  Classical  regularity results in Sobolev spaces  for the heat equation \cite{MR161099} show that $u^{\eta}_n\in L^p_\textup{loc}((0,\infty); H^{2,p}(\rr^N))$   and $\partial_tu\in L^p_\textup{loc}((0,\infty); L^{p}(\rr^N))$  for all $p\in(1,\infty)$. It follows that  $u^{\eta}_n$ belongs to  the parabolic H\"older spaces $C^{\delta,2\delta}(Q)$ for some $\delta>0$, \cite[Section~8.5, p.~117]{MR1406091}. The regularity of $f_\eta$ guarantees that $\partial_{x_N}(f_\eta(u^\eta_n))$ also belongs to $C^{\delta,2\delta}(Q)$ and then regularity results for the inhomogeneous heat equation
\cite[Theorem 8.7.3, p.~123]{MR1406091} show that $u\in C^{1+\delta,2+2\delta}(Q)$, so $u^\eta_n$ is a classical solution.

Besides, \cite[Lemma B.17]{JFL} gives that $u_n^\eta\rightarrow u_n$ in $C([0,\infty);L^1(\rr^N))$ as $\eta\to0^+$.  Hence,  in view of the previous step, $u_n^\eta\rightarrow u$ in $C([\tau,\infty);L^1(\rr^N))$ for all $\tau>0$.  A similar result holds for $\ou_n^\eta$, which is defined analogously.

\smallskip
\noindent (iv) \emph{The primitives are bounded solutions of the Hamilton-Jacobi equation.} It is easily checked that
\[
v_n^\eta (t,x',x_N):=\int_{-\infty}^{x_N} u^\eta_n(t,x',y_N)\,{\rm d}y_N\quad\text{for all }(t,x',x_N)\in (0,\infty)\times\rr^{N-1}\times\rr
\]
is a nonnegative classical solution of the Hamilton-Jacobi equation
\begin{equation}
\label{HJ-nonlocal}
  \partial_tv- \Delta v =f_\eta(\partial_{x_N}v) \quad\text{in }Q.
\end{equation}
Let us prove that it is bounded. Indeed, since $u^\eta_n\in C([0,\infty);L^1(\rr^N))$ is nonnegative,
\begin{equation}
	\label{vcomparedtow}
v_n^\eta(t,x',x_N)\leq \int_{\rr} u^\eta_n(t,x',y_N)\,{\rm d}y_N=:w_\eta^n(t,x'),
\end{equation}
where the function $w_\eta^n$ belongs to $C([0,\infty);L^1(\rr^{N-1}))$ and solves \normalcolor
\begin{equation}
\label{eq:wneps}
\partial_tw-\Delta_{x'} w=0\quad\text{in $(0,\infty)\times\rr^{N-1}$,}\qquad w(0)=M\Gamma_{N-1}(1/n)\quad\text{on $\rr^{N-1}$.}
\end{equation}
The solution of this problem is $w_\eta^n(t)=M\Gamma_{N-1}(t)\ast \Gamma_{N-1}(1/n)$. Hence,
$$
\|w_\eta^n(t)\|_{L^\infty(\rr^{N-1})}\leq M\| \Gamma_{N-1}(1/n)\|_{L^\infty(\rr^{N-1})}\leq Mc_n
$$
for some positive constant $c_n$, and $v^\eta_n$ is  bounded. The same results, with the same bound, hold for $\overline v^{\eta}_n$, which is defined analogously.

\smallskip
\noindent (v) {\emph{Comparison of the traces.}} We claim that, for all~$(x',x_N)\in \rr^{N-1}\times\rr$,
\begin{equation}
\label{eq:comparison.integrals}
\int_{-\infty}^{x_N-2r} u_{0,n}^r(x',y_N)\,{\rm d}y_N\leq \int_{-\infty}^{x_N} \ou_{0,n}^r(x',y_N)\,{\rm d}y_N \leq \int_{-\infty}^{x_N+2r} u_{0,n}^r(x',y_N)\,{\rm d}y_N.
\end{equation}
We prove only the first inequality since the second one can be obtained similarly.
By construction, $u_{0,n}^r(x',\cdot)$ is supported in $(-r,r)$. Thus, if $x_N<r$,  the left-hand-side term vanishes, and,  hence the inequality is true since $\ou_{0,n}^r$ is nonnegative. If $x_N>r$, using  the support of
$\ou_{0,n}^r(x',\cdot) $ and~\eqref{mass.u_0n},
\begin{align*}
\label{}
  \int_{-\infty}^{x_N-2r} u_{0,n}^r(x',y_N)\,{\rm d}y_N&\leq  \int_{-\infty}^{\infty} u_{0,n}^r(x',y_N)\,{\rm d}y_N=
  M\Gamma_{N-1}(1/n,x')=\int_{-\infty}^{\infty} \ou_{0,n}^r(x',y_N)\,{\rm d}y_N\\
  &=\int_{-\infty}^{r} \ou_{0,n}^r(x',y_N)\,{\rm d}y_N =\int_{-\infty}^{x_N} \ou_{0,n}^r(x',y_N)\,{\rm d}y_N,
\end{align*}
and the inequality is also true.

On the other hand, since $u^\eta_n\in C([0,\infty);L^1(\rr^N))$, then $u^\eta_n(t)\rightarrow u_{0,n}^r$  in $L^1(\rr^N)$ as $t\rightarrow 0^+$. Thus,
$$
v_n^\eta (t,x',x_N)\rightarrow \int_{-\infty}^{x_N} u_{0,n}^r(x',y_N)\,{\rm d}y_N\quad\text{in }L^1_{x'} (\rr^{N-1};L_{x_N}^\infty(\rr))\text{ as }t\rightarrow 0^+,
$$
and, \normalcolor  hence, a.e.~in $\rr^{N-1}\times\rr$ as well. Similar arguments hold for $ \overline v^\eta_n$. Therefore,~\eqref{eq:comparison.integrals} translates into
\begin{equation}
\label{claim.v}
  v^\eta_n(0+,x',x_N-2r)\leq \overline v^\eta_n(0+,x',x_N)\leq v^\eta_n(0+,x',x_N+2r) \quad\text{for a.e} \ (x',x_N)\in \rr^{N-1}\times\rr.
\end{equation}

\smallskip
\noindent(vi) \emph{Comparison of the primitives.} Let us now show that the inequalities~\eqref{claim.v} for the traces imply
\begin{equation}
\label{claim.v.2}
   v^\eta_n(t,x',x_N-2r)\leq \overline v^\eta_n(t,x',x_N)\leq v^\eta_n(t,x',x_N+2r)\quad\text{for all }(t,x',x_N)\in (0,\infty)\times \rr^{N-1}\times\rr.
\end{equation}
We only prove the first inequality, since the proof of the second is analogous. To this aim,  we define
\[
g(t,x',x_N):=v^\eta_n(t,x',x_N-2r)- \overline v^\eta_n(t,x',x_N)\quad\text{for all }(t,x',x_N)\in (0,\infty)\times\rr^{N-1}\times\rr.
\]
By step (iv), $g$ is uniformly bounded by $2Mc_n$. Moreover, $g\in C^{1+\delta, 2+2\delta} (Q)$ and satisfies
\begin{eqnarray}\label{eq.g}
\partial_t g-\Delta g= a(t,x) \partial_{x_N}g\quad\text{in }Q,\qquad \lim _{t\to0^+}g(t)\leq 0\quad\text{a.e. on }\rr^{N},
\end{eqnarray}
where $|a(t,x)|\leq c_{\eta,q}$.

Following~\cite[Theorem 3]{MR2422079}, we consider $\psi:\rr^N\rightarrow \rr_+$ smooth such that  $\psi(x)= 0$ for $|x|\leq 1$, $\psi(x)>\|g\|_{L^\infty(Q)}$ for $|x|\geq2$, with $\psi, \nabla\psi$ and $D^2\psi$ bounded in $\rr^N$. With this choice we introduce, for all $\beta>0$, $\psi_\beta(x)=\psi(\beta x)$. It follows that   $\psi_\beta(x)>\|g\|_{L^\infty(Q)}$ for $|x|\geq2/\beta$, and $\nabla \psi_\beta$, $D^2\psi_\beta$, and $\ml\psi_\beta$ all go to zero uniformly on $\rr^N$ when $\beta \rightarrow 0^+$. Let  $\tilde\delta>0$ and $h(t,x):=g(t,x)-\tilde\delta t-\psi_\beta(x)$. Then $h^+(t) \in W^{1,1}(\rr^N)$ is compactly supported in the ball $|x|<2/\beta$ and
  satisfies $\nabla h^+(t)=\nabla h(t)\mathds{1}_{(h(t)>0)}$  for all $t>0$.  Hence, by~\eqref{eq.g},
\begin{align*}
	\frac 12\frac{\rm d}{{\rm d}t}&\int_{\rr^N }(h^+(t))^2=\int_{\rr^N}h^+(t)\partial_th(t)=	\int_{\rr^N}h^+(t)(\partial_tg(t)-\tilde\delta)\\
&=\int_{\rr^N } h^+(t)  (\Delta g(t)+a \partial_{x_N}g(t)-\tilde\delta)\\
	&=\int_{\rr^N } h^+(t)\big(\Delta  h(t)+a \partial_{x_N}h(t)\big)+
	\int_{\rr^N } h^+(t)\big(\Delta  \psi_\beta+a \partial_{x_N} \psi_\beta-\tilde\delta\big).
\end{align*}
Observe that for $\beta$ small enough, $\beta <\beta(\tilde\delta)$, we have
\[
\big|\Delta  \psi_\beta+a \partial_{x_N} \psi_\beta\big|<\tilde\delta \quad\text{in }\rr^N.
\]

Hence, for all $t>0$, Young's inequality gives
\begin{align*}
	\frac 12\frac{\rm d}{{\rm d}t}\int_{\rr^N }(h^+(t))^2&\leq -\int_{\rr^N}|\nabla h^+|^2+	\int_{\rr^N }a h^+(t)\partial_{x_N}h^+(t)\\
	&\leq-\int_{\rr^N}|\nabla h^+|^2+ \frac{c^2_{n,q}}{4}\int_{\rr^N}  (h^+(t))^2 +
	\int_{\rr^N}|\partial_{x_N}h^+(t)|^2\leq  \frac{c^2_{n,q}}{4}  \int_{\rr^N} (h^+(t))^2.
\end{align*}
Then for all $0<s<t$ we have
\[
\int_{\rr^N }(h^+(t))^2\leq \int_{\rr^N }(h^+(s))^2 +  \frac{c^2_{n,q}}{4}\int_s^t  \int_{\rr^N }(h^+)^2.
\]
Observing that
\[
\lim_{s\rightarrow 0^+} \int_{\rr^N }(h^+(s))^2= \lim_{s\rightarrow 0^+} \int_{|x|<2/\beta}\big((g(s)-\tilde\delta s-\psi_\beta\big)^+)^2 =0,
\]
we conclude that $h^+\equiv 0$, by Gr\"onwall's inequality, so that
\[
g(t,x)\leq \tilde\delta t+\psi_\beta(x) \quad \text{for all }t>0,\text{ all }x\in \rr^N,\text{ and all }\beta<\beta(\tilde\delta).
\]
This implies that $g(t,x)\leq \tilde\delta t$ on $\{|x|<1/\beta\}$ and, \normalcolor  hence, on $\rr^N$ after letting $\beta\rightarrow 0^+$. Finally, the result follows by taking the limit $\tilde\delta\rightarrow 0^+$.

\smallskip
\noindent(vii) \emph{Conclusion.} Inequalities~\eqref{claim.v.2} can be written as
\[
\int_{-\infty}^{x_N-2r} u_n^\eta(t,x',y_N)\,{\rm d}y_N\leq \int_{-\infty}^{x_N} \ou_n^\eta(t,x',y_N)\,{\rm d}y_N \leq \int_{-\infty}^{x_N+2r} u_n^\eta(t,x',y_N)\,{\rm d}y_N.
\]
Since  $u_n^\eta\rightarrow u_n$ in $C([\tau,\infty);L^1(\rr^N))$ for all $\tau>0$, we deduce that $u_n^\eta(t,x',\cdot)\rightarrow u_n (t,x',\cdot)$ in $L^1(\rr)$ for a.e.~$x'\in \rr^{N-1}$ and all $t\geq \tau$. Then, letting $\eta\rightarrow 0$ we get, up to a subsequence,
\[
\int_{-\infty}^{x_N-2r} u_n (t,x',y_N)\,{\rm d}y_N\leq \int_{-\infty}^{x_N} \ou_n (t,x',y_N)\,{\rm d}y_N \leq \int_{-\infty}^{x_N+2r} u_n (t,x',y_N)\,{\rm d}y_N\quad\text{for a.e. }x'.
\]
On the other hand, since   $u_n(t)\rightarrow u(t)$ in $ L^1(\rr^N)$,  up to a subsequence, $u_n (t,x',\cdot)\rightarrow u (t,x',\cdot)$ in $L^1(\rr)$ for a.e.~$x'\in \rr^{N-1}$. We conclude that,  for all $t\geq\tau$,
\[
\int_{-\infty}^{x_N-2r} u  (t,x',y_N)\,{\rm d}y_N\leq \int_{-\infty}^{x_N} \ou  (t,x',y_N)\,{\rm d}y_N \leq \int_{-\infty}^{x_N+2r} u  (t,x',y_N)\,{\rm d}y_N\quad\text{for a.e. }x'.
\]
Letting $r\to0^+$, we deduce that $u(t,x',x_N)=\ou (t,x',x_N)$ for a.e.~$(x',x_N)\in \rr^{N-1}\times\rr$ and all $t\geq \tau$, whence the uniqueness result.

\noindent\emph{The reduced equation.} The proof  is similar to the case $\mathcal{L}=-\Delta$. Let us comment on the points where some care has to be taken.
In step (i) we still have that $v$ defined by~\eqref{integrated.u} is the unique fundamental solution with mass $M$ of  $\partial_t v-\Delta_{x'} v=0$. In step (iii) we need not only to regularize the nonlinearity but also to add a diffusion term in the $x_N$ variable
\begin{equation}\label{reg'}
\partial_tu -\Delta_{x'}u-\varepsilon \partial_{x_Nx_N}^2u+\partial_{x_N} (f_\eta(u)) =0\quad\textup{in }Q.
\end{equation}
Its solutions satisfy $u^{\eps,\eta}\rightarrow u$ in $C([\tau,\infty);L^1(\rr^N))$ for all $\tau>0$ as $\eps,\eta\to0$.
The primitives in step (iv) satisfy the Hamilton-Jacobi equation
\[
\partial_t v-\Delta_{x'} v-\varepsilon \partial_{x_Nx_N}^2v =f_\eta(\partial_{x_N}v) \quad \text{in }Q,
\]
instead of~\eqref{HJ-nonlocal}. We can bound these primitives in terms of a function $w_\eps^n$ defined as in~\eqref{vcomparedtow}, which solves the same problem~\eqref{eq:wneps} as the function $w_\eps^n$ that was defined for the other case, and that hence has the same bound. The proof of step (vi) also works nicely, once we observe that $\int_{\rr^N} h^+ (\Delta_{x'}h)\leq 0$ and the term $-\eps |\partial_{x_N} h^+|^2 $ helps to absorb the extra term appearing when applying Young's inequality.
\end{proof}

%%%%%%%%%%%%%%%%%%%%%%%%%%%%%%%%
\section{Source-type solutions have a sign}
\label{sect-SignChanges} \setcounter{equation}{0}

We now prove the uniqueness and existence of source-type solutions having a tail-control.

\begin{proof}[Proof of Theorem \ref{thm:uniqueness.no.sign}] We follow the main steps in \cite{Carpio1996} but adapting them to the concept of entropy solutions.

\smallskip
\noindent(i) \emph{$u^+$ and $u^-$ satisfy \eqref{eq:sub.entropy.inequality}.} Since $u$ is an entropy solution it follows that $u^+$ and $u^-$ satisfy the regularity in Definition \ref{def.sub-entropySolution} away from $t=0$, i.e. $u^{\pm}\in  L^\infty ([\tau,\infty);L^\infty(\rr^N))$ and $\nabla u^{\pm}\in L^2_{\textup{loc}}([\tau,\infty)\times \rr^N)$ for all $\tau>0$. Let us prove that both of them are entropy subsolutions. Since $u$ is an entropy solution, it satisfies inequality \eqref{eq:sub.entropy.inequality}. Clearly, the $0$-function satisfies the same identity.  Using Proposition \ref{results.subsolutions}
we obtain that $u^+=\max\{u,0\}$ is a subsolution. In the case of the negative part we first observe that, since $f(u)$ is an odd nonlinearity,  $-u$ is again an entropy solution of the same problem with different initial data.  Using again the $0$-function, we obtain that $u^-=\max\{-u,0\}$ is an entropy subsolution.

We emphasize that in the case of one dimensional scalar conservation laws the fact that $u^+$ and $u^-$ are entropy subsolutions has been proved in \cite[Lemma~1.1]{MR735207}.

\smallskip
\noindent(ii) \emph{Initial data of $u^+$ and $u^-$.} Since $u(t)\rightarrow M\delta_0$ as $t\rightarrow 0$ it follows that for any sequence $(t_n)_{n\geq 1}\rightarrow 0$, there exists a subsequence $(t_j)_{j\geq 1}\rightarrow 0$ and two  nonnegative finite measures $\mu$ and $\nu$ such that $M\delta_0=\mu -\nu$ and
\[
u^+(t_j)\rightarrow \mu, \ u^{-}(t_j)\rightarrow \nu
\]
in the sense of measures. The tail control  implies that the two measures must be Dirac measures supported at $x=(x',x_N)=0$.

\smallskip
\noindent(iii) \emph{Comparison of $u^+,u^-$ with entropy solutions.} Let $m_j\in C_\textup{c}^\infty(\rr^N)$ be a sequence of nonnegative functions converging to $M\delta_0$. Let us denote by $h_j$ and $g_j$ the entropy solutions of problem \eqref{eq.fLgeneral} with initial data $h_j(0)=u^+(t_j)$ and $g_j(0)=u^{-}(t_j)+m_j$ belonging to $L^1(\rr^N)\cap L^\infty(\rr^N)$, respectively. The uniqueness result in Theorem \ref{lem.UniquenessPropertiesEntropy}(b) guarantees the uniqueness of the two solutions. Observe that $v_j(t):=u^+(t_j+t)$   is an entropy subsolution of \eqref{eq.fLgeneral} with the same initial data as $h_j$. It follows from Proposition \ref{results.subsolutions} that $v_j(t)\leq h_j(t)$. A similar result holds for $u^-.$ Hence
\[
0\leq u^+(t+t_j)\leq h_j(t), \ 0\leq u_j^{-}(t+t_j)\leq g_j(t)\quad \text{for a.e $t>0$.}
\]

\smallskip
\noindent(iv) \emph{Estimates for $h_j$ and $g_j$.}
Since $u\in L^\infty((0,\infty);L^1(\rr^N))$ it follows that
\[
\int_{\rr^{N}} u^\pm(t_j)\leq \int _{\rr^N} |u(t_j,x)|dx\leq \|u\|_{L^\infty((0,\infty);L^1(\rr^N))}=A<\infty
\]
and hence the sequences $(h_j(0))_{j\geq 1}$ and $(g_j(0))_{j\geq 1}$ are uniformly bounded in $L^1(\rr^N)$.

Let us recall some results in \cite{JFL} for nonnegative solutions $w$ of  \eqref{eq.fLgeneral} with initial data $0\leq w_0 \in L^1(\rr^N)\cap L^\infty(\rr^N)$. They conserve the mass, and when $1-\frac 1N<q\leq 2$, we have
hyperbolic estimates (see the proof of Proposition 4.4 in \cite{JFL} with $\alpha=2$ and $\lambda=1$):

\noindent$\bullet$ \textup{(Time decay of $L^p$-norm)} For all $p\in[1,\infty]$ and a.e. $t>0$,
\begin{equation}\label{eq.HEDecayOfLpNorm}
\|w (t)\|_{L^p(\rr^N)}\lesssim t^{-\frac{N+1}{2q}(1-\frac 1p)}\|w_0\|_{L^1(\rr^N)}^{\frac 1q(1-\frac 1p)}.
\end{equation}

\noindent$\bullet$ \textup{(Energy estimate)} For all $0<\tau<T<\infty$,
\begin{equation}\label{eq.HEEnergyEstimate}
\int_\tau^T \|\mathcal{L}^{\frac{1}{2}}w(t,\cdot)\|_{L^2(\rr^N)}^2dt\lesssim \tau ^{-\frac {N+1}{2q}}\|w_0\|_{L^1(\rr^N)}^{1/q}.
\end{equation}

\noindent$\bullet$ \textup{(Estimate on the time derivative)} For all bounded domains $\Omega\subset \rr^N$,
\begin{equation}\label{eq.HEEstimateTimeDerivative}
\| \partial_t w\|_{L^2((\tau,T);H^{-1}(\Omega))}\lesssim C(\tau,\|w_0\|_{L^1(\rr^N)}).
\end{equation}

We now apply the above estimates to $w_0=h_j(0)$. A similar argument will be used for $g_j(0)$.

 Let us first consider the case when $\ml=-\Delta$.
Since $(h_j(0))_{j>0}$ is uniformly bounded in $L^1(\rr^N)$, the above estimates \eqref{eq.HEEnergyEstimate}, \eqref{eq.HEEstimateTimeDerivative}   allow us to use the classical compactness argument of Aubin-Lions-Simon \cite[Theorem 5]{Simon}. Thus,
 up to a subsequence,  $h_j\rightarrow h$ in $L^2_{\textup{loc}}(Q)$,  $h_j\rightarrow h$ a.e. in $Q$, and $h_j\rightarrow h$ in $C_{\textup{loc}}((0,\infty); H^{-\eps}(\Omega))$  for every $\eps>0$ and $\Omega\subset \rr^N$ a bounded domain.

When $\ml=-\Delta_{x'}$, the compactness is a little bit more difficult since
 the energy estimate \eqref{eq.HEEnergyEstimate} is only a partial gradient:
\[
\int_\tau^T \|\nabla_{x'}w(t,\cdot)\|_{L^2(\rr^N)}^2dt\lesssim \tau ^{-\frac {N+1}{2q}}\|w_0\|_{L^1(\rr^N)}^{1/q}.
\]
However, by a Hoff/Ole\u{\i}nik-type estimate, one can deduce a uniform shift estimate in the  direction $x_N$ (see \cite[Lemma 3.8]{JFL}):
\begin{align}\label{translation.N}
\int_{|x'|<R', |x_N|<R}&|w(t,x+(0,\xi_N))-w(t,x)|dx\\
&
\lesssim |\xi_N|R'^{N-1}\Big(R^{-1}t^{-1-\frac{N+1}2\frac{q-2}q}\|w_0\|_{L^1(\rr^N)}^{\frac {2-q}q}+t^{-\frac 1q\frac{N+1}2} \|w_0\|_{L^1(\rr^N)}^{\frac 1q} \Big).  \nonumber
\end{align}
The compactness result in \cite[Lemma C.2]{JFL} with $\alpha=2$ gives the desired convergence of $(h_j)_{j\geq 1}$.

\medskip

The compactness obtained above  allow us to pass to the limit    in the entropy inequality \eqref{eq:entropy.inequality}
\[
\iint_Q \Big( |h_j-k|\partial_t\phi + \sgn(h_j-k)\big(f(h_j)-f(k)\big)\partial_{x_N}\phi-|h_j-k|(\mathcal{L}\phi) \Big)\geq0
\]
to obtain that the limit point $h$ also satisfies \eqref{eq:entropy.inequality}.

Let us now identify the initial data that $h$ takes in the sense of measures. We use the fact that $h_j$ conserves the mass and it is also a very weak solution. After a classical approximation argument we deduce that for any $\psi\in C_\textup{c}^\infty(\rr^N)$ we have
\begin{align*}
\Big|\int_{\rr^N} h_j(t,x)\psi(x)\,{\rm d}x-\int _{\rr^N}h_j(0)\psi(x)\,{\rm d}x \Big|\leq RHS&:=
\int_0^t|h_j| |\mathcal{L} \psi|+ \int_{\rr^N} |h_j|^q |\partial_{x_N}\psi |\\
&\leq tA\|\mathcal{L} \psi \|_{L^\infty(\rr^N)}  +\int_0^t \int_{\rr^N} |h_j|^q |\partial_{x_N}\psi |.
\end{align*}
Since $q \in(1-\frac 1N,1)$, we use the $L^1$-bound in Lemma \ref{lem.UniquenessPropertiesEntropy} (iv) to get
\[
\int_0^t \int_{\rr^N} |h_j|^q |\partial_{x_N}\psi |\leq \int_0^t \|h_j\|_{L^1(\rr^N)}^q \|\partial_{x_N}\psi\|_{L^{1/(1-q)}(\rr^N)} \leq C(A) \|\partial_{x_N}\psi\|_{L^{1/(1-q)}(\rr^N)}  t,
\]
and then the $RHS$  goes to zero as $t\rightarrow0^+$ uniformly in $j$.
Letting $j\rightarrow\infty$ it follows that for all $\psi\in C_\textup{c}^\infty(\rr^N)$
\[
\int_{\rr^N } h(t,x)\psi(x)\,{\rm d}x \rightarrow \langle \mu,\psi\rangle \quad \textup{as $t\rightarrow 0^+$.}
\]
To prove the above convergence for all $\psi\in C_\textup{b}(\rr^N)$, we use the tail control in Lemma \ref{lem.FutherPropEntropy}(c) and the above considerations to get
\begin{equation}\label{eq.HETailControl}
\begin{split}
\int_{|x|>2R} h_j(t,x)\,{\rm d}x&\lesssim  \int_{|x|>R} u^+(t_j)\,{\rm d}x+ \frac{At}{R^2}+ C(A)\frac{t}{R^{1-N(q-1)}}.
\end{split}
\end{equation}
Letting $j\rightarrow \infty$ and using that $u^+(t_j)$ converges to a multiple of Dirac mass we get
\[
\int_{|x|>2R} h(t,x)\,{\rm d}x\lesssim    \frac{At}{R^2}+ C(A)\frac{t}{R^{1-N(q-1)}}
\]
and then again $\int_{\rr^N} h(t)\psi\rightarrow \langle \mu,\psi\rangle$ as $t\rightarrow 0$ for all $\psi\in C_\textup{b}(\rr^N)$.

\smallskip
\noindent(v) \emph{Application of uniqueness for nonnegative solutions.}
It follows that $h$ is a nonnegative entropy solution of equation \eqref{eq.fLgeneral} with initial data a Dirac delta denoted by $\mu$.
A similar argument holds for the sequence  $(g_j)_{j\geq 1}$ and then its limit point $g$ is again an entropy solution of equation \eqref{eq.fLgeneral} but with an initial data  $ \nu + M\delta_0=\mu$. From the uniqueness of the nonnegative solutions with Dirac delta initial data we deduce that $h=g$.

\noindent(vi) \emph{Conclusion.} Using that $h=g$ and letting   $t_j\rightarrow 0^+$ and then $t\rightarrow 0^+$ we get $\nu=0$, $\mu=M\delta_0$, $u^-=0$, $u=u^+$ and the uniqueness result for nonnegative solutions applies again. The full details are given in \cite{Carpio1996}.
\end{proof}

\begin{remark}
\label{rk:q.2.3}
All the  estimates in Step (iv) in the proof of Theorem \ref{thm:uniqueness.no.sign} are based on the $L^1$--$L^p$ hyperbolic estimates \eqref{eq.HEDecayOfLpNorm}. When $\ml=-\Delta$ one can also use the parabolic $L^1$--$L^p$ estimates (see \cite[Lemma 3.5]{JFL} with $\alpha=2$) which are available in the range $q>1-1/N$:
\[
\|w(t)\|_{L^p(\rr^N)}\lesssim t^{-\frac N2(1-\frac 1p)}\|w_0\|_{L^1(\rr^N)}\quad\textup{for all $p\in[1,\infty]$ and   $t>0$}.
\]
This then gives
\[
\int_{|x|>2R} h(t,x)\,{\rm d}x\lesssim    \frac{At}{R^2}+ C(A)\frac{t^{1-\frac N2(q-1)}}R.
\]
Note that the right-hand side goes to zero as $t\rightarrow 0^+$ for all $q\in(1,1+2/N)$. This has already been observed in dimension one for nonnegative solutions in \cite[Section~4]{EVZArma}. We are then able to apply the   arguments in Step (iv) also when $N=1$ and $q\in (2,3)$, a case which was not treated in \cite{Carpio1996}.
\end{remark}

We now proceed to the existence proof.
\begin{proof}[Proof of Theorem \ref{thm:existence}]
Let us consider $m_j=M j^{-N}\varphi(x/j)$, $j\geq 1$, with $\varphi$ a nonnegative smooth function with compact support and mass one. Applying the same arguments as above  to $(g_j)_{j\geq 1}$ we obtain the required existence.
\end{proof}

%%%%%%%%%%%%%%%%%%%%%%%%%%%%%%%%%%%%%%%%%%%%%%%%%%%%%%%%%%%%%%%%%%%%%%%%%%%%%%%%%%
\section{Application to large-time behaviour}
\label{sec:AsymptoticBehaviour}
\setcounter{equation}{0}

In this section, we will treat the large-time behaviour of signed solutions of \eqref{eq.fLgeneral} when $q\in(1-1/
N,1)$.

\begin{proof}[Proof of Theorem \ref{thm:long.time.behaviour}]
\noindent\textup{(i)} \emph{The case $\mathcal{L}=-\Delta$.} We introduce the one-parameter family of functions
$$
u_\lambda(t,x',x_N):=\lambda^\gamma u(\lambda t, \lambda^{1/2}x', \lambda^\beta x_N),
$$
where the exponents $\gamma,\beta>0$ are such that:\\
\noindent$\bullet$ The mass of the solutions remains constant.\\
\noindent$\bullet$ The decay in the $L^\infty$-norm remains independent of $\lambda$.\\
The conservation of mass imposes
\[
\gamma=\frac{N-1}{2} + \beta.
\]

Since $1-\frac{1}{N}<q <1$, we are in the subcritical case, and we rely on hyperbolic estimates. In particular we have \eqref{eq.HEDecayOfLpNorm} with $p=\infty$.
  To keep this estimate invariant for $u_\lambda$, we set
\[
\beta=\frac{N+1-q(N-1)}{2q}.
\]
The large time behaviour for $u$ will follow from the behaviour as $\lambda\to\infty$ of $u_\lambda$ for some fixed time $t$. Note that $u_\lambda$ is an entropy solution of
\begin{equation}\label{eq.u_lambdaSubcritical}
\partial_t u_\lambda-\Delta_{x'}u_\lambda-\lambda^{1-2\beta}\partial_{x_Nx_N}^2u_\lambda+\partial_{x_N}(u_\lambda^q)=0\quad\textup{in $Q$},\qquad u_\lambda(0)=u_{0,\lambda}\quad\textup{in $\rr^N$},
\end{equation}
where $u_{0,\lambda}(x',x_N):=\lambda^\gamma u_0(\lambda^{1/2}x',\lambda^\beta x_N)$ and $1-2\beta<0$. We therefore expect that the solution of \eqref{eq.u_lambdaSubcritical} converges to the source-type solution of $\partial_t u-\Delta_{x'}u+\partial_{x_N}(u^q)=0$ as $\lambda\to\infty$.

To do so, we need several results concerning the sequence $(u_\lambda)_{\lambda>1}$: time decay of $L^p$-norm (cf. \eqref{eq.HEDecayOfLpNorm}), conservation of mass (cf. Lemma \ref{lem.UniquenessPropertiesEntropy}(a)(v)), energy estimate (cf. \eqref{eq.HEEnergyEstimate}), local estimate on the $x_N$-derivative (cf. \eqref{translation.N}), estimate on the time derivative (cf. \eqref{eq.HEEstimateTimeDerivative}), and tail control (cf. \eqref{eq.HETailControl}).  This is exactly Proposition 4.4 in \cite{JFL} with $\alpha=2$.  All the mentioned properties can then be transferred to $u_\lambda$ by scaling.

The proof of Theorem \ref{thm:long.time.behaviour} then follows the steps of the proof of the subcritical case for Theorem 1.3 in \cite{JFL} with $\alpha=2$. Compactness in the form of Lemma C.2 in \cite{JFL} with $\alpha=2$ is mainly a consequences of  the energy estimate for the operator $-\Delta_{x'}$ together with the local estimate on the $x_N$-derivative. This and the tail control  yield that there is a subsequence such that $u_\lambda(t)\to U(t)$ in $L^1(\rr^N)$ for a.e. $t>0$ and $u_\lambda(t,x)\to U(t,x)$ for a.e. $(t,x)\in Q$ as $\lambda\to\infty$. The limit equation has already been identified, but a rigorous discussion on the stability of the entropy formulation also follows as in \cite{JFL} with $\alpha=2$. Finally, the identification of the initial data heavily depends on the deduced tail control.

The uniqueness (cf. Theorem \ref{thm:uniqueness.non-negative}) of the limit equation then guarantees that the whole sequence $(u_\lambda)_{\lambda>0}$ converges to the limit $U$, and we obtain Theorem \ref{thm:long.time.behaviour}.

\smallskip
\noindent\textup{(ii)} \emph{The case $\mathcal{L}=-\Delta_{x'}$.} We will have to start by replacing $\mathcal{L}=-\Delta_{x'}$ by $\mathcal{L}=-\Delta_{x'}-\varepsilon \partial_{x_Nx_N}^2$ in \eqref{eq.fLgeneral}. The same considerations as above imply that $u_\lambda^\varepsilon$ solves
\begin{equation}\label{eq.u_lambdaSubcritical2}
\partial_t u_\lambda^\varepsilon-\Delta_{x'}u_\lambda^\varepsilon-\lambda^{1-2\beta}\varepsilon\partial_{x_Nx_N}^2u_\lambda^\varepsilon+\partial_{x_N}((u_\lambda^\varepsilon)^q)=0\quad\textup{in $Q$},\qquad u_\lambda^\varepsilon(0)=u_{0,\lambda}\quad\textup{in $\rr^N$}.
\end{equation}
Again, $1-2\beta<0$, and we therefore expect the same limit problem as above when $\lambda\to\infty$.

Repeating the previous argument, $u_\lambda^\varepsilon$ also satisfies the mentioned properties above (uniformly in $\varepsilon$ and $\lambda$). Again, Lemma C.2 in \cite{JFL} provides compactness of $(u_\lambda^\varepsilon)_{\lambda>0,\varepsilon>0}$. Up to a subsequence, we have that $u_{\lambda_k}^{\varepsilon_k}(t)\to U(t)$ in $L^1(\rr^N)$ for a.e. $t>0$ and $u_{\lambda_k}^{\varepsilon_k}(t,x)\to U(t,x)$ for a.e. $(t,x)\in Q$ as $k\to\infty$. The proof can then be finished as before.
\end{proof}

%%%%%%%%%%%%%%%%%%%%%%%%%%%%%%%%%%%%%%%%%%%%%%%%%%%%%%%%%%%%%%%%
\subsection*{Acknowledgements}

L. I.  Ignat received funding from Romanian Ministry of Research, Innovation and Digitization, CNCS - UEFISCDI, project number PN-III-P1-1.1-TE-2021-1539, within PNCDI III.

F. Quir\'os supported by grants CEX2019-000904-S, PID2020-116949GB-I00, and RED2018-102650-T, all of them funded by MCIN/AEI/10.13039/501100011033, and by the Madrid Government (Comunidad de Madrid – Spain) under the multiannual Agreement with UAM in the line for the Excellence of the University Research Staff in the context of the V PRICIT (Regional Programme of Research and Technological Innovation).

%%%%%%%%%%%%%%%%%%%%%%%%%%%%%%%%%%%%%%%%%%%%%%%%%%%%%%%%%%%%%%%%%%%%%%%%%%%%%%%%%%
%References

%\bibliographystyle{abbrv}
%\bibliography{biblio}

\end{document}